\documentclass[11pt, reqno]{amsart}
\usepackage{amsmath, amsthm, amsfonts, amssymb, amsbsy, bigstrut, graphicx, enumerate,  upref, longtable, comment, booktabs, array, caption, subcaption}

\usepackage{pstricks}
\usepackage{times}

\usepackage[numbers, sort&compress]{natbib}

\usepackage[breaklinks]{hyperref}

\newcommand{\tf}{\tilde{f}}

\newtheorem{thm}{Theorem}[section]
\newtheorem{lmm}[thm]{Lemma}
\newtheorem{cor}[thm]{Corollary}
\newtheorem{prop}[thm]{Proposition}

\theoremstyle{definition}

\newcommand{\cov}{\mathrm{Cov}}

\newcommand{\ee}{\mathbb{E}}

\newcommand{\mf}{\mathcal{F}}

\newcommand{\ml}{\mathcal{L}}

\newcommand{\cq}{\mathcal{Q}}

\newcommand{\pp}{\mathbb{P}}

\newcommand{\rr}{\mathbb{R}}

\newcommand{\var}{\mathrm{Var}}
\newcommand{\ve}{\varepsilon}

\newcommand{\zz}{\mathbb{Z}}

\newcommand{\fpar}[2]{\frac{\partial #1}{\partial #2}}

\numberwithin{equation}{section}

\usepackage[utf8]{inputenc}
\usepackage{amsmath}
\usepackage{amsfonts}
\usepackage{bbm}
\usepackage{mathtools}
\usepackage{tikz}
\usetikzlibrary{decorations.pathreplacing}
\usepackage{amsthm}
\usepackage[english]{babel}
\usepackage{fancyhdr}
\usepackage{amssymb}
\usepackage{enumitem}
\usepackage{qtree}
\usepackage{mathrsfs}

\renewcommand{\tilde}{\widetilde}

\begin{document}

\title[Convergence of directed polymers to deterministic KPZ]{Weak convergence of directed polymers to deterministic KPZ at high temperature}
\author{Sourav Chatterjee}
\address{Departments of Mathematics and Statistics, Stanford University}
\email{souravc@stanford.edu}
\thanks{Research partially supported by NSF grant DMS-1855484}
\keywords{Directed polymer, KPZ, random surface, scaling limit}
\subjclass[2010]{82C41, 60G60, 39A12}



\begin{abstract}
It is shown that when $d\ge 3$, the growing random surface generated by the $(d+1)$-dimensional directed polymer model at sufficiently high temperature, after being smoothed by taking microscopic local averages,  converges to a solution of the deterministic KPZ equation in a suitable scaling limit. 
\end{abstract}

\maketitle


\section{The directed polymer model}
Take any $d\ge 1$. Let $\beta$ be a positive real number, let $\xi=(\xi_{t,x})_{t\in \zz_{\ge1}, x\in \zz^d}$ be a collection of i.i.d.~random variables and let $g:\zz^d \to \rr$ be a function. The $(d+1)$-dimensional directed polymer model at inverse temperature $\beta$, in the random environment $\xi$, and initial condition $g$, defines a growing random surface $f:\zz_{\ge0}\times \zz^d \to \rr$ as follows. For any $t$ and $x$, let $\cq(t,x)$ be the set of all nearest-neighbor paths in $\zz^d$ of length $t$ that end at $x$. A typical element $q\in \cq(t,x)$ is a $(t+1)$-tuple $(q_0,q_1,\ldots,q_t)$ of elements of $\zz^d$, where $q_t=x$, and $|q_i-q_{i+1}|=1$ for each $i$, where $|\cdot|$ denotes Euclidean norm. Let $f(0,x)=g(x)$ for each $x\in \zz^d$, and for $t\in \zz_{>0}$ and $x\in \zz^d$, let
\begin{align}\label{polymerform}
f(t,x) &= \frac{1}{\beta} \log \biggl[\sum_{q\in \cq(t,x)} \exp\biggl(\beta g(q_0) + \beta \sum_{i=1}^t\xi_{i,q_i}\biggr)\biggr]. 
\end{align}
The directed polymer model has a long and storied history in probability and statistical physics. For recent surveys of this vast literature, see~\cite{bateschatterjee20, comets17}. The model is supposed to converge to a scaling limit given by the Kardar--Parisi--Zhang (KPZ) equation~\cite{kardaretal86}. Formally, the KPZ equation is given by
\begin{align}\label{kpzeq}
\partial_t h = \nu \Delta h + \frac{\lambda}{2} |\nabla h|^2 + \sqrt{D} \dot{W},
\end{align}
where $h(t,x)$ is the height of a continuum random surface at time $t$ and location $x$, $\dot{W}$ is a random field known as space-time white noise, and $\nu$, $\lambda$, and $D$ are  nonnegative real-valued parameters. In dimension one, the meaning of the KPZ equation, as well as the convergence of the directed polymer model to a KPZ limit, are now well-understood --- see \cite{chatterjee21} for a brief survey. 

For $d\ge 3$, it has been shown recently in~\cite{lygkoniszygouras22} that when $\beta$ is sufficiently small, the discrete surface defined by~\eqref{polymerform} converges to a Gaussian field upon suitable centering and scaling. A similar result was proved for $d=2$, when $\beta$ decays as $t\to\infty$ like $(\log t)^{-1/2}$, in the earlier paper \cite{caravennaetal17}. These Gaussian fields are the same as the ones coming from recent attempts at constructing solutions to the KPZ equation in $d\ge 2$ in the so-called {\it Edwards--Wilkinson regime}~\cite{coscoetal20, mukherjeeetal16, magnenunterberger18, dunlapetal20, chatterjeedunlap20, caravennaetal20, cometsetal19, cometsetal20, gu20}. In this regime, the gradient-squared term of the KPZ equation \eqref{kpzeq} does not appear (that is, $\lambda = 0$), which reduces the KPZ equation to the stochastic heat equation (SHE) with additive noise. 





This paper proposes a different notion of KPZ convergence for directed polymers in $d\ge 3$ and sufficiently small $\beta$, where the limit is the {\it deterministic KPZ equation}
\begin{align}\label{kpzeq2}
\partial_t h = \nu \Delta h + \frac{\lambda}{2} |\nabla h|^2,
\end{align}
which is just \eqref{kpzeq} with $D= 0$. The coefficients $\nu$ and $\lambda$ are determined by $\beta$ and $d$. The limit is obtained by smoothing the discrete surface \eqref{polymerform} by locally averaging it over balls of small but growing radius. This notion of convergence to KPZ, by local smoothing of a discrete random surface, seems to be a new idea that has not appeared in the prior literature. One can say that this is convergence to KPZ `on average', or just `weak convergence'. It would be interesting to see if there are other models of surface growth that  exhibit similar behavior (for example, those discussed in \cite{chatterjee21b}).



\section{Results}\label{resultsec}
Fix $d\ge 3$, as before. Let  $\xi=(\xi_{t,x})_{t\in \zz_{\ge1}, x\in \zz^d}$ be a field of i.i.d.~random variables, which we will refer to as the `noise variables'. Let $g:\rr^d\to \rr$ be a Lipschitz function. For each $\ve >0$, define $g_\ve :\zz^d \to \rr$ as $g_\ve(x) := g(\ve x)$. Let $f_\ve$ be the growing random surface generated by the directed polymer model with initial data $f_\ve(0,\cdot) = g_\ve(\cdot)$. Let $\ml$ denote the set of all probability measures on $\rr$ that are push-forwards of the standard Gaussian measure under some Lipschitz map. We assume that the law of the noise variables belongs to the class $\ml$. Then, in particular, $\xi_{t,x}$ has a finite moment generating function $m(\beta) := \ee(e^{\beta \xi_{t,x}})$. Define
\[
\mu(\beta) := \frac{m(2\beta)}{m(\beta)^2}.
\]
Let $\rho_d$ be the probability that a simple symmetric random walk started at the origin in $\zz^d$ returns to the origin at least once. Since $d\ge 3$, $\rho_d < 1$. In particular, $\mu(0) = 1< 1/\rho_d$. Let $\beta_0$ be the supremum of all nonnegative $\beta$ such that $\mu(\beta) < 1/\rho_d$. It is easy to see that $\mu$ is a continuous function, which implies that $\beta_0> 0$. The range $(0,\beta_0)$ is known as the ``$L^2$ regime'' of sufficiently high temperatures. For a justification of this nomenclature, as well as the importance of this regime and its various occurrences in the literature, see~\cite{lygkoniszygouras22}. 

To state the main result about the scaling limit of $f_\ve$, we first need to identify a constant related to the model, whose existence at high temperature is established by the following proposition.
\begin{prop}\label{etaprop}
When $\beta \in (0,\beta_0)$, the following limit exists and is finite:
\[
\eta(\beta) := \lim_{t\to\infty} \ee\biggl[\log \biggl\{\frac{1}{(2dm(\beta))^{t}} \sum_{q\in \cq(t,0)} \exp\biggl( \beta \sum_{i=1}^t\xi_{i,q_i}\biggr)\biggr\}\biggr]. 
\]
\end{prop}
The quantity $\eta(\beta)$ is the `second order term' in the asymptotic expression for the log partition function at high temperature as $t\to\infty$, the first order term being $t \log (2dm(\beta))$. It turns out that a multiple of this second order term has to be subtracted off when renormalizing the surface to obtain the deterministic KPZ limit below. It is not immediately clear if $\eta(\beta)$ admits an explicit calculation.

For $x\in \zz^d$ and $r>0$, let $B(x,r)$ denote the set of all $y\in \zz^d$ such that $|x-y|\le r$. Let $\{r_\ve\}_{\ve >0}$ be any collection of positive real numbers tending to infinity as $\ve \to 0$, such that $r_\ve = o(\ve^{-1})$. For $t\in \rr$, let $[t]$ denote the greatest integer $\le t$. For $x = (x_1,\ldots,x_d)\in \rr^d$, let $[x]$ denote the vector $([x_1],\ldots,[x_d])$. For $t\in \rr_{\ge 0}$, $x\in \rr^d$, and $\ve>0$, define the rescaled values 
\begin{align}\label{tvedef}
t_\ve := [\ve^{-2} t], \ \ \ x_\ve := [\ve^{-1} x].
\end{align}
Define the smoothed, rescaled and renormalized surface $\tf^{(\ve)}: \rr_{\ge 0} \times \rr^d\to \rr$ as 
\begin{align}\label{tfvedef}
\tf^{(\ve)}(t,x) := \frac{1}{|B(x_\ve, r_\ve)|} \sum_{y\in B(x_\ve, r_\ve)} f_\ve(t_\ve, y) - \frac{t_\ve}{\beta}\log(2d m(\beta)) - \frac{\eta(\beta)}{\beta}.
\end{align}
The following theorem is the main result of this paper.
\begin{thm}\label{polymerthm}
Let all notations and conditions be as above, with $d\ge 3$. Suppose that $\beta \in (0,\beta_0)$. Let $h$ be the Cole--Hopf solution of the deterministic KPZ equation \eqref{kpzeq2} with $\nu=1/2d$, $\lambda = \beta/d$, and initial condition $g$. This is defined as 
\[
h(t,x) = \frac{1}{\beta} \log \ee[\exp(\beta g(\sqrt{t/d} Z+ x))],
\]
where $Z$ is a $d$-dimensional standard Gaussian random vector. Then for any $t> 0$ and $x\in \rr^d$, the random variable $\tf^{(\ve)}(t,x)$ converges in $L^2$ to $h(t,x)$ as $\ve \to 0$.  
\end{thm}
Theorem \ref{polymerthm} proves convergence of $\tf^{(\ve)}(t,x)$ to the deterministic KPZ limit $h(t,x)$ for each fixed $(t,x)$. This distinguishes it from much of the recent literature, which deal with convergence towards Edwards--Wilkinson limits on average ---  that is, after integrating the field against some test function~\cite{caravennaetal20, coscoetal20,dunlapetal20,gu20,lygkoniszygouras22}.  From this perspective, the pointwise statement of Theorem \ref{polymerthm} has some resemblance to the results of \cite{cometsetal19} and \cite{cometsetal20} which give pointwise laws of large numbers and central limit theorems for the difference between a solution of the KPZ equation with given initial data and a stationary solution of the KPZ equation. 

The main idea behind the proof of Theorem \ref{polymerthm} is the following. We show that $f_\ve(t_\ve,y)$ splits as the sum of a random quantity that has no dependence on the initial condition $g$, and an almost deterministic (i.e., highly concentrated) quantity that approximately equals $h(t,x)$ with high probability. The random part, moreover, is a function only of noise variables located very close to $(t_\ve,y)$, so that averaging the random parts over $y\in B(x_\ve, r_\ve)$ results in an approximately deterministic value that has no dependence on $g$. The renormalization term in \eqref{tfvedef} is this deterministic value.

One can also consider the rescaled and renormalized surface without smoothing, defined as
\[
f^{(\ve)}(t,x) := f_\ve(t_\ve, x_\ve) - \frac{t_\ve}{\beta}\log(2d m(\beta)) - \frac{\eta(\beta)}{\beta}.
\]
It is a simple consequence of Theorem \ref{polymerthm} that $f^{(\ve)}$ converges to $h$ in the a certain weak sense as $\ve \to 0$. 
\begin{cor}\label{polymercor}
Take any continuous function $\phi:\rr^d\to \rr$ with compact support. Then for any $t>0$, as $\ve \to 0$, the random variable $\int f^{(\ve)}(t,x) \phi(x)dx$ converges in $L^2$ to $\int h(t,x)\phi(x)dx$.
\end{cor}
Results analogous to Corollary \ref{polymercor} have been proved for a model of continuous directed polymers in \cite{mukherjeeetal16, coscoetal20}. The weak limit in that model turns out to be the heat equation rather than the deterministic KPZ equation, because limit is obtained for the partition function rather than the log partition function. It is probable that the weak limit of the log partition function in the continuous model is also deterministic KPZ. It would be interesting to see if some analogue of Theorem \ref{polymerthm} can be proved in the continuous setting. 

A natural open problem is to extend the above results to the entire high temperature regime $(0,\beta_c)$, where $\beta_c$ is the smallest critical inverse temperature of the model (see \cite{bateschatterjee20} for the precise definition; it is known that $\beta_c\ge \beta_0$). It is possible that Theorem \ref{polymerthm} and Corollary \ref{polymercor} are valid for all $\beta\in (0,\beta_c)$, especially in view of \cite[Theorem 1.1]{bateschatterjee20}, which shows that the endpoint distribution of the polymer is not localized when $\beta \le \beta_c$. It is possible that delocalization of the endpoint is sufficient for the proof of Theorem \ref{polymerthm} instead of the stronger condition $\beta \le \beta_0$. This possibility is bolstered by the fact that a result similar to Corollary \ref{polymercor} for the partition function (instead of the log partition function) in the entire high temperature regime in a related continuous model has recently been established in~\cite{coscoetal20}.

Another interesting direction is to extend Theorem \ref{polymerthm} to dimension two. It is unclear what the appropriate analogue of Theorem \ref{polymerthm} would be in $d=2$. 

Investigating the behavior of second order fluctuations is a natural question. For example, in Theorem \ref{polymerthm}, one can ask if for some explicit $\gamma>0$, $\ve^{-\gamma}(\tf^{(\ve)}(t,x) - h(t,x))$ converges in law to a non-degenerate limiting distribution as $\ve \to 0$. It is unclear whether this might hold for fixed $(t,x)$, or after averaging against a smooth test function. A similar question can be asked in the context of Corollary~\ref{polymercor}, about the rescaled difference $\ve^{-\gamma}\int (f^{(\ve)}(t,x) -h(t,x))\phi(x)dx$. For this, the results of \cite{coscoetal20,dunlapetal20,lygkoniszygouras22,mukherjeeetal16} indicate that the  exponent $\gamma$ could be $(d-2)/2$, because these papers show that $\ve^{-(d-2)/2}\int (f^{(\ve)}(t,x) -\ee(f^{(\ve)}(t,x)))\phi(x)dx$ has a non-degenerate limiting distribution in a number of related settings.

Finally, as pointed out by one of the referees (and in analogy with the results of \cite{coscoetal20, mukherjeeetal16}), it is likely that in Theorem \ref{polymerthm} and Corollary~\ref{polymercor}, if one takes a suitable weak limit of the partition function instead of the log partition function, one would obtain the deterministic heat equation instead of the deterministic KPZ equation. It would be interesting to see if this is true, under a similar local averaging framework.

The rest of the paper is devoted to the proofs of the above results. Many technical ideas in the proofs are similar to ideas from various recent papers, such as \cite{caravennaetal17, caravennaetal20, lygkoniszygouras22}. The main new idea is the formulation of the scaling limit, presented above, that leads to the deterministic KPZ limit.

\section{Concentration inequalities}
We will use a technique from the classical theory of concentration of measure to control the lower tail probabilities of the partition functions of the polymer model. This technique was first used in \cite{carmonahu02}. For a recent application, see~\cite{caravennaetal20}. A complete account is given below for the reader's convenience.

First, we need the following classical result, known as the Gaussian concentration inequality (see \cite[Appendix A]{chatterjee14}). 
\begin{thm}[Gaussian concentration inequality]\label{gaussconcthm}
Let $X$ be an $n$-dimensional standard Gaussian random vector and let $f:\rr^n\to \rr$ be a Lipschitz  function with Lipschitz constant $L$. Then for any $t\ge 0$,
\[
\pp(|f(X)-\ee (f(X))| \ge t) \le 2e^{-t^2/2L^2}. 
\]
\end{thm}
Let $\ml(K)$ be the set of all probability measures on $\rr$ that are obtained as push-forwards of the standard Gaussian measure under a Lipschitz map with Lipschitz constant $\le K$. Let $\ml_n(K)$ denote the set of the product measures on $\rr^n$ whose marginal distributions are all in $\ml(K)$. We will say that a random vector is in class $\ml_n(K)$ if its law is in $\ml_n(K)$. The following result is an easy corollary of the Gaussian concentration inequality.
\begin{cor}\label{lipconc}
Let $X$ be an $n$-dimensional random vector in class $\ml_n(K)$ and let $f:\rr^n\to \rr$ be a Lipschitz  function with Lipschitz constant $L$. Then for any $t\ge 0$,
\[
\pp(|f(X)-\ee (f(X))| \ge t) \le 2e^{-t^2/2L^2K^2}. 
\]
\end{cor}
\begin{proof}
If $X = (X_1,\ldots,X_n)$, we can write $X_i = g_i(Y_i)$ where $Y_1,\ldots, Y_n$ are i.i.d.~standard Gaussian random variables and $g_1,\ldots, g_n$ are Lipschitz functions with Lipschitz constant $K$. Then $f(X)$, as a function of $Y = (Y_1,\ldots, Y_n)$, has Lipschitz constant $\le LK$. The claimed inequality now follows from the Gaussian concentration inequality.
\end{proof}
A corollary of the above corollary is the following result, which is a standard idea from concentration of measure~\cite{ledoux01}.
\begin{cor}\label{conccor}
Let $X$ be an $n$-dimensional random vector in class $\ml_n(K)$. Let $A$ be a closed subset of $\rr^n$. Let $D := \inf\{\|X-x\|: x\in A\}$ and $p := \pp(X\in A)$. Then for any $t\ge 0$,
\begin{align*}
\pp(D \ge K\sqrt{2\log(2/p)}+ t) &\le 2e^{-t^2/2K^2}. 
\end{align*}
\end{cor}
\begin{proof}
It is not hard to see that $D$ is a Lipschitz function of $X$ with Lipschitz constant $1$. Therefore by Corollary \ref{lipconc},
\begin{align}\label{ppd}
\pp(|D-\ee(D)|\ge t) \le 2e^{-t^2/2K^2}. 
\end{align}
Since $A$ is a closed set, the above inequality implies that
\begin{align*}
p&= \pp(X\in A)= \pp(D=0)\\
&\le \pp(|D-\ee(D)|\ge \ee(D)) \le 2e^{-\ee(D)^2/2K^2},
\end{align*}
which gives
\[
\ee(D)\le K\sqrt{2\log (2/p)}.
\]
Thus, again appealing to \eqref{ppd}, we get the desired result.
\end{proof}
Using Corollary \ref{conccor}, we obtain the following variant of \cite[Proposition 1.6]{ledoux01}.
\begin{prop}\label{tailprop}
Let $X$ be an $n$-dimensional random vector in class $\ml_n(K)$. Let $f:\rr^n\to \rr$ be a continuously differentiable convex function. Take any $a\in \rr$ and $b>0$, and let $p:= \pp(f(X)\ge a, \, |\nabla f(X)|\le b)$. Then for any $t\ge 0$,
\begin{align*}
\pp(f(X)\le a- Kb\sqrt{2\log(2/p)} -t)\le 2e^{-t^2/2b^2K^2}.
\end{align*}
\end{prop}
\begin{proof}
By convexity,
\[
f(y)-f(x)\le \nabla f(y)\cdot(y-x)
\]
for any $x, y\in \rr^n$. Thus, if $f(y)\ge a$ and $|\nabla f(y)|\le b$, we have that for any $x$,
\[
f(x) \ge a - b|x-y|.
\]
Let $A$ be the set of all $y$ such that $f(y)\ge a$ and $|\nabla f(y)|\le b$. Since $f$ is continuously differentiable, $A$ is a closed set. Let $p := \pp(X\in A)$ and let $D$ be the distance of $X$ from $A$. Then by the above inequality,
\begin{align*}
f(X) \ge a - b D. 
\end{align*}
Thus, for any $t\ge 0$,
\begin{align*}
\pp(f(X)\le a- Kb\sqrt{2\log(2/p)} - t)  &\le \pp(a-bD \le a- Kb\sqrt{2\log(2/p)} - t)\\
&= \pp(D\ge K\sqrt{2\log(2/p)} + t/b).
\end{align*}
The claimed result now follows by Corollary \ref{conccor}.
\end{proof}

\section{Random walk estimates}\label{rwsec}
In this section, $C,C_1,C_2,\ldots$ will denote positive constants depending only on $d$, whose values may change from line to line. (Here $d$ is the dimension fixed in Section \ref{resultsec}. Recall that $d\ge 3$.)  If a constant depends on some additional parameter $\theta$, we will denote it by $C(\theta)$. 

Let $\{S_n\}_{n\ge 0}$ be a simple symmetric random walk on $\zz^d$, starting at the origin. We will now collect several estimates for this walk that will be useful later. First, recall the well-known fact that for each $n$,
\begin{align}\label{rwest}
\sup_{x\in \zz^d} \pp(S_n = x)\le \frac{C}{n^{d/2}}.
\end{align}
(The simplest way to show this is by Fourier transforms. For example, see \cite[Lemma 18.3]{chatterjee21}.) On the other hand, it is a simple consequence of Hoeffding's inequality~\cite{hoeffding63} for sums of independent and uniformly bounded random variables that for any $x$ and $n$, 
\begin{align}\label{rwest2}
\pp(S_n=x) \le C_1 e^{-C_2|x|^2/n}.
\end{align}
Another easy consequence of Hoeffding's inequality that we will use later is that for any $\theta\ge 0$ and $n\ge 0$,
\begin{align}\label{expsn}
\ee(e^{\theta |S_n|}) \le C_1e^{C_2\theta^2n}.
\end{align}
By Chebychev's inequality, 
\[
\pp(|S_n|\le 2\sqrt{n})\ge 1 - \frac{\ee|S_n|^2}{4n} = \frac{3}{4}.
\]
Thus,  for any $\theta \ge 0$ and any $n$,
\begin{align}\label{expminus}
\ee(e^{-\theta |S_n|})\ge e^{-2\theta \sqrt{n}} \pp(|S_n|\le 2\sqrt{n}) \ge \frac{3e^{-2\theta \sqrt{n}}}{4}.
\end{align}
Combining \eqref{rwest} and \eqref{rwest2}, and letting $n_0:= a (2+|x|)^2/\log(2+|x|)$ for a sufficiently small number $a$, we get 
\begin{align}
\pp(S_n = x\text{ for some } n) &\le C_1 \sum_{n\le n_0}  e^{-C_2|x|^2/n} + C_1\sum_{n> n_0}n^{-d/2}\notag \\
&\le C_1 n_0 e^{-C_2 |x|^2/n_0} + C_1 n_0^{-(d-2)/2} \notag\\
&\le C_1 (\log (2+|x|))^{C_2}(2+ |x|)^{2-d}. \label{greeneq}
\end{align} 
Recall that $\rho_d$ is the probability that $S_n$ returns to the origin at least once. Let $\{S_n'\}_{n\ge 0}$ be another $d$-dimensional simple symmetric random walks started at the origin, independent of $\{S_n\}_{n\ge 0}$. It is easy to see that $\{S_n-S_n'\}_{n\ge 0}$ is a Markov chain with the same law as the Markov chain $\{S_{2n}\}_{n\ge 0}$. Since a simple symmetric random walk can return to its starting point only at even times, this shows that if 
\begin{align}\label{ndef}
N_n := |\{0\le k\le n: S_k = S_k'\}|, \ \ \ N_\infty := |\{k\ge 0: S_k=S_k'\}|,
\end{align}
then we have
\begin{align}\label{npmf}
\pp(N_\infty =k) = \rho_d^{k-1}(1-\rho_d)
\end{align}
for each $k\ge 1$. In particular, $\ee(e^{a N_\infty})$  is finite if $0< a < \log(1/\rho_d)$. 

Next, suppose that the two random walks are started from two different locations $x$ and $y$. Let $N_n^{x,y}$ and $N_\infty^{x,y}$ be the analogues of $N_n$ and $N_\infty$ in this situation. Let $\tau^{x,y}$ be the smallest $n$ such that $S_n=S_n'$. If there is no such $n$, let $\tau^{x,y}=\infty$. Let
\begin{align}\label{kappadef}
\kappa(x,y) := \pp(\tau^{x,y}<\infty).
\end{align}
Note that by \eqref{greeneq}, and the fact that $\{S_n-S_n'\}_{n\ge 0}$ behaves like $\{S_{2n}\}_{n\ge 0}$ started at $x-y$, we get 
\begin{align}
\kappa(x,y) &= \pp(S_{2n} = 0 \text{ for some } n) \notag\\
&\le C_1 (\log (2+|x-y|))^{C_2}(2+ |x-y|)^{2-d}.\label{kappabd}
\end{align}
If $\tau^{x,y}=\infty$, then $N_\infty^{x,y} = 0$ and $N_n^{x,y}=0$ for all $n$. On the other hand, if $\tau^{x,y}$ is finite, the strong Markov property of $\{S_n-S_n'\}_{n\ge 0}$ implies that $N_n^{x,y}$ has the same law as $N_n$ and $N_\infty^{x,y}$ has the same law as $N_\infty$. Thus, by \eqref{greeneq}, we get that for any  function $f:\zz_{\ge 0}\to [0,\infty)$ with $f(0)=0$, 
\begin{align}
\ee(f( N_\infty^{x,y})) &= \ee(f(N_\infty))\kappa(x,y). \label{covprep}
\end{align}
We will use the above facts on several occasions. We will also use the following lemma.
\begin{lmm}\label{kappalmm}
Let $\{S_n\}_{n\ge 0}$ and $\{S_n'\}_{n\ge 0}$ be independent simple symmetric random walks on $\zz^d$ starting at the origin. Take any $\alpha>0$. Then for all $n$,
\begin{align*}
\ee[\kappa(S_n,S_n')^\alpha] &\le  C_1(\alpha) (\log n)^{C_2(\alpha)} n^{-(d-2)\alpha'/2},
\end{align*}
where $\alpha' = \min\{\alpha, d/(d-2)\}$. 
\end{lmm}
\begin{proof}
Since $\kappa$ is uniformly bounded by $1$, we have $\kappa(x,y)^\alpha \le \kappa(x,y)^{\alpha'}$ for all $x,y$. So let us assume without loss of generality that $\alpha \le d/(d-2)$. By the facts that $\kappa(x,y)\le 1$ and $\kappa(x,y) = \kappa(0,x-y)$ for all $x,y$, and that $S_n-S_n'$ has the same law as $S_{2n}$, we get
\begin{align*}
\ee[\kappa(S_n,S_n')^\alpha] &\le \sum_{z\, :\, |z|\le \sqrt{n}\log n}\kappa(0,z)^\alpha \pp(S_{2n}=z) + \pp(|S_{2n}|\ge \sqrt{n}\log n).
\end{align*}
By Hoeffding's inequality,
\begin{align*}
\pp(|S_{2n}|\ge \sqrt{n}\log n) \le 2e^{-C(\log n)^2}. 
\end{align*}
On the other hand, by \eqref{rwest}, $\pp(S_{2n}=z)\le Cn^{-d/2}$ for all $z$. Therefore, by \eqref{kappabd}, 
\begin{align*}
&\sum_{z\, :\, |z|\le \sqrt{n}\log n}\kappa(0,z)^\alpha \pp(S_{2n}=z) \\
&\le C_1(\alpha) (\log n)^{C_2(\alpha)}n^{-d/2} \sum_{z\, :\, |z|\le \sqrt{n}\log n}(1+ |z|)^{(2-d)\alpha} \\
&\le C_1(\alpha) (\log n)^{C_2(\alpha)}n^{-d/2} \sum_{0\le r\le \sqrt{n}\log n} r^{d-1} (1+r)^{(2-d)\alpha}\\
&\le C_1(\alpha) (\log n)^{C_2(\alpha)} n^{-(d-2)\alpha/2}. 
\end{align*}
(The assumption that $\alpha \le d/(d-2)$ was used in the last step.) The proof is now completed by combining the three displays above.
\end{proof}

\section{Discrete approximation of the Cole--Hopf solution}
In this section, $C, C_1,C_2,\ldots$ will denote positive constants whose values may depend only on $d$, $\beta$ and $g$. The values may change from line to line.

Let $h$ be the Cole--Hopf solution to the deterministic KPZ equation, defined in the statement of Theorem \ref{polymerthm}. For a proof that this is indeed a solution, see \cite{chatterjee21}. We will need to work with a discrete approximation of $h$, defined as follows. For any $\ve >0$, define $G_\ve:\zz_{\ge 0}  \times \zz^d$ as 
\begin{align*}
G_\ve(t,x) := \ee(e^{\beta g_\ve(S_t + x)}), 
\end{align*}
where, as in Section \ref{rwsec}, $\{S_n\}_{n\ge0}$ denotes a simple symmetric random walk in $\zz^d$, started at the origin. Note that the definition of $G_\ve$ depends on $\beta$ and $g$, but we are suppressing that for notational simplicity. The following lemmas about $G_\ve$ will be useful.
\begin{lmm}\label{glim}
As $\ve \to 0$, $\beta^{-1} \log G_\ve(t_\ve, x_\ve) \to h(t,x)$ for any $(t,x)\in \rr_{\ge 0}\times \rr^d$. 
\end{lmm}
\begin{proof}
Note that 
\[
G_\ve(t_\ve, x_\ve) =  \ee(e^{\beta g_\ve(S_{t_\ve} + x_\ve)}) =  \ee(e^{\beta g(\ve S_{t_\ve} + \ve x_\ve)}).
\]
By the multivariate central limit theorem, we know that as $\ve \to 0$, $t^{-1/2}_\ve S_{t_\ve}$ converges in distribution to $d^{-1/2}Z$, where $Z$ is a $d$-dimensional standard Gaussian random vector. Also, as $\ve \to 0$, $\ve x_\ve \to x$ and $\ve t^{1/2}_\ve \to \sqrt{t}$. Thus, as $\ve \to 0$, $\ve S_{t_\ve}+\ve x$ converges in distribution to $\sqrt{t/d}Z + x$. Consequently, $e^{\beta g(\ve S_{t_\ve} + \ve x_\ve)}$ converges in distribution to $e^{\beta g(\sqrt{t/d}Z + x)}$. 

Next, note that by the Lipschitz property of $g$, the inequality \eqref{expsn}, and the facts that $\ve|x_\ve|\le C_1+C_2|x|$ and $\ve^2 t_\ve \le C_1 +C_2 t$,
\begin{align*}
\ee(e^{2\beta g_\ve(S_{t_\ve} + x_\ve)}) &\le C_1\ee(e^{C_2\ve (|S_{t_\ve}| + |x_\ve|)}) \\
&\le C_1 e^{C_2(\ve^2 t_\ve+\ve|x_\ve|)} \le C_3e^{C_4(t+|x|)}.
\end{align*}
From this uniform boundedness of the second moment of $e^{\beta g(\ve S_{t_\ve} + \ve x_\ve)}$, and the distributional convergence proved in the previous paragraph, it is now easy to see that as $\ve \to 0$, 
\[
G_\ve(t_\ve,x_\ve) = \ee(e^{\beta g(\ve S_{t_\ve} + \ve x_\ve)}) \to \ee(e^{\beta g(\sqrt{t/d}Z + x)}),
\]
which completes the proof of the lemma. 
\end{proof}
\begin{lmm}\label{glip}
For any $t$, $x$ and $y$,
\[
|G_\ve(t,x)-G_\ve(t,y)| \le C_1 \ve|x-y| e^{C_2(\ve|x|+\ve|y|+\ve^2 t)}.
\]
\end{lmm}
\begin{proof}
By the Lipschitz property of $g$, 
\begin{align*}
|G_\ve(t,x)-G_\ve(t,y)| &\le \ee|e^{\beta g_\ve(S_t + x)} - e^{\beta g_\ve(S_t + y)}|\\
&\le C \ve|x-y|  \ee(e^{\beta g_\ve(S_t + x)}+e^{\beta g_\ve(S_t + y)})\\
&\le C_1 \ve|x-y|e^{C_2\ve(|x|+|y|)} \ee(e^{C_2\ve |S_t|}). 
\end{align*}
By \eqref{expsn}, this completes the proof. 
\end{proof}
\begin{lmm}\label{gtbound}
For any $s$, $t$ and $x$,
\begin{align*}
|G_\ve(s,x)-G_\ve(t,x)|&\le C_1\ve |t-s|^{1/2} e^{C_2(\ve|x|+\ve^2s +\ve^2t)}. 
\end{align*}
\end{lmm}
\begin{proof}
Note that by the Lipschitz property of $g$,
\begin{align*}
&|G_\ve(s,x)-G_\ve(t,x)| \\
&\le \ee|e^{\beta g_\ve(S_s + x)} - e^{\beta g_\ve(S_t + x)}|\\
&\le \beta \ee[|g_\ve(S_s + x) - g_\ve(S_t + x)|(e^{\beta g_\ve(S_s + x)} + e^{\beta g_\ve(S_t + x)})]\\
&\le C \ve\ee[|S_s-S_t| (e^{\beta g_\ve(S_s + x)} + e^{\beta g_\ve(S_t + x)})].
\end{align*}
Now applying the Cauchy--Schwarz inequality, the fact that $\ee[(S_s-S_t)^2]=|t-s|$, and the same technique as in the proof of Lemma \ref{glip} to obtain upper bounds on the expected values of the exponential terms, we get the desired result.
\end{proof}
\begin{lmm}\label{gbounds}
We have
\[
C_1 e^{-C_2\ve(|x| + \sqrt{t})} \le G_\ve (t,x) \le  C_3 e^{C_4(\ve|x|+\ve^2 t)}. 
\]
\end{lmm}
\begin{proof}
By the Lipschitz property of $g$, 
\begin{align*}
 \ee(e^{-C_1 - C_2 \ve (|S_t|+|x|)}) \le G_\ve(t,x) \le \ee(e^{C_3 + C_4\ve(|S_t|+|x|)}).
\end{align*}
The proof is now easily completed by invoking  \eqref{expsn} and \eqref{expminus}.
\end{proof}

\section{Concentration of the height function}\label{concheightsec}
Let $\sigma$ be the law of the noise variables. In this section, $C,C_1,C_2,\ldots$ will denote positive constants depending only on $d$, $\beta$, $\sigma$  and $g$, whose values may change from line to line. If a constant depends on some additional parameter $\theta$, we will denote it by $C(\theta)$.

Fix some $\ve >0$. Recall the function $\mu(\beta) = m(2\beta)/m(\beta)^2$, where $m$ is the moment generating function of the noise variables. Note that $\mu(\beta)\ge 1$ for any $\beta$. Take any $t\in \zz_{\ge1}$ and $x\in \zz^d$. For $q\in \cq(t,x)$, let
\[
Y_q(t,x) := \frac{ \exp\bigl(\beta \sum_{i=1}^t \xi_{i,q_i}\bigr)}{m(\beta)^t},
\]
and let
\begin{align}\label{ydef}
Y(t,x) := \frac{1}{(2d)^t}\sum_{q\in \cq(t,x)} Y_q(t,x).
\end{align}
Also, define
\[
Z_\ve(t,x) := \frac{1}{(2d)^t}\sum_{q\in \cq(t,x)} e^{\beta g_\ve(q_0)} Y_q(t,x). 
\]
Finally, let $F(t,x) := \log Y(t,x)$ and $F_\ve(t,x) := \log Z_\ve(t,x)$. In this section we will prove two lemmas that will show that $F$ and $F_\ve$ have fluctuations of order $1$, with exponentially decaying tails. These technical results will be used later in the proof.
\begin{lmm}\label{techlmm1}
If $\mu(\beta) < 1/\rho_d$, then for any $\theta > 0$, $\ee(e^{-\theta F(t,x)})\le C(\theta)$.
\end{lmm}
\begin{proof}
Note that $\ee(Y_q(t,x)) = 1$ for any $t$, $x$ and $q$, and so $\ee(Y(t,x))=1$. Next, note that for any $t$, $x$ and $y$, and any $q\in \cq(t,x)$ and $r\in \cq(t,y)$, 
\begin{align}
\ee(Y_q(t,x)Y_r(t,y)) &= \frac{1}{m(\beta)^{2t}}\prod_{i=1}^t \ee(e^{\beta (\xi_{i,q_i}+ \xi_{i, r_i})}) \notag\\
&= \mu(\beta)^{|q\cap r|},\label{yqyr}
\end{align}
where $q\cap r$ denotes the set of $i$ such that $q_i=r_i$. Thus, 
\begin{align*}
\ee(Y(t,x)^2) &= \frac{1}{(2d)^{2t}}\sum_{q, r\in \cq(t,x)} \ee(Y_q(t,x) Y_r(t,y))\\
&=  \frac{1}{(2d)^{2t}}\sum_{q, r\in \cq(t,x)} \mu(\beta)^{|q\cap r|}\\
&= \ee(\mu(\beta)^{N_t}),
\end{align*}
where $N_t$ is defined as in \eqref{ndef}. By \eqref{npmf}, the assumption that $\mu(\beta) < 1/\rho_d$, and the observation that $\mu(\beta)\ge 1$, this shows that
\begin{align}\label{y2bound}
\ee(Y(t,x)^2)\le \ee(\mu(\beta)^{N_\infty}) \le C.
\end{align}
Also, as noted above, $\ee(Y(t,x)) = 1$. Thus, by \eqref{y2bound} and the Paley--Zygmund second moment inequality,
\begin{align}\label{yhalf}
\pp(Y(t,x)\ge 1/2)\ge C >0.
\end{align}
Next, note that for any $1\le s\le t$ and $y\in \zz^d$,
\begin{align*}
\fpar{}{\xi_{s,y}} F(t,x) &= \frac{1}{Y(t,x)}\fpar{}{\xi_{s,y}} Y(t,x)\\
&= \frac{1}{Y(t,x)(2d)^t}\sum_{q\in \cq(t,x)}\fpar{}{\xi_{s,y}} Y_q(t,x)\\
&= \frac{\beta }{Y(t,x)(2d)^t}\sum_{q\in \cq(t,x), \, q(s)=y} Y_q(t,x).
\end{align*}
Let $\nabla F$ denote the gradient of $F$, considering $F$ as a function of $(\xi_{s,y})_{1\le s\le t, y\in \zz^d}$. The above calculation shows that
\begin{align*}
|\nabla F(t,x)|^2 &= \frac{\beta^2 }{Y(t,x)^2(2d)^{2t}}\sum_{1\le s\le t}\sum_{y\in \zz^d}\sum_{\substack{q,r\in \cq(t,x), \\  q(s)=r(s)=y}} Y_q(t,x)Y_r(t,x)\\
&=  \frac{\beta^2 }{Y(t,x)^2(2d)^{2t}}\sum_{q,r\in \cq(t,x)} |q\cap r|Y_q(t,x)Y_r(t,x).
\end{align*}
Thus, by \eqref{yqyr}, the assumption that $\mu(\beta) < 1/\rho_d$, and the fact that $\mu(\beta)\ge 1$, we get
\begin{align*}
\ee|\nabla Y(t,x)|^2 &= \ee|Y(t,x)\nabla F(t,x)|^2 \\
&= \beta^2 \ee(N_t \mu(\beta)^{N_t}) \le \beta^2\ee(N_\infty \mu(\beta)^{N_\infty}) \le C. 
\end{align*}
Combining this with \eqref{yhalf}, we see that there are constants $a\in \rr$ and $b>0$, depending only on $d$, $\beta$ and $\sigma$, such that 
\begin{align*}
&\pp(F(t,x) \ge a, \, |\nabla F(t,x)|\le b) \\
&\ge \pp(F(t,x)\ge a) - \pp(F(t,x)\ge a, \, |\nabla F(t,x)|> b)\\
&=  \pp(F(t,x)\ge a) - \pp(Y(t,x)\ge e^{a}, \, |\nabla Y(t,x)|> bY(t,x))\\
&\ge \pp(F(t,x)\ge a) - \pp(|\nabla Y(t,x)|> be^a)\\ 
&\ge  C >0. 
\end{align*}
It is a standard fact, easily verifiable by computing second derivatives, that $F(t,x)$ is a convex function of the noise variables. Thus, by Proposition \ref{tailprop} and the above inequality, we have
\begin{align}\label{ftxbd}
\pp(F(t,x)\le -C_1 - u) \le 2 e^{-C_2 u^2}. 
\end{align}
This shows that for any $\theta >0$, 
\begin{align*}
\ee(e^{-\theta F(t,x)}) &= \int_0^\infty \pp(e^{-\theta F(t,x)} \ge v) dv\\
&\le \int_0^\infty \pp(F(t,x) \le -\theta^{-1} \log v) dv\le C(\theta). 
\end{align*}
This completes the proof of the lemma.
\end{proof}
\begin{lmm}\label{techlmm2}
If $\mu(\beta) < 1/\rho_d$, then for any $\ve \in (0,1)$ and $\theta>0$, 
\[
\ee(e^{ -\theta F_\ve(t,x)})\le \psi(\ve|x|+\ve\sqrt{t}+\ve^2t),
\]
where $\psi:(0,\infty)\to (0,\infty)$ is an increasing continuous function that depends only on $\beta$, $d$, $\sigma$, $g$ and $\theta$.
\end{lmm}
\begin{proof}
The proof is similar to the proof of Lemma \ref{techlmm1}, with minor adjustments. Throughout, we will denote by $\psi_1,\psi_2,\ldots$ increasing, positive, continuous functions of $\ve|x|+\ve\sqrt{t}+\ve^2t$, which depend only on $\beta$, $d$, $\sigma$, $g$, and $\theta$. 

First, note that by the Lipschitz property of $g$, we have $|g_\ve(x)-g_\ve(y)|\le C \ve |x-y|$ for all $x,y$, and $|g_\ve(x)|\le C_1+C_2\ve|x|$ for all $x$. Thus,
\begin{align}
\ee(Z_\ve(t,x)^2) &= \frac{1}{(2d)^{2t}}\sum_{q, r\in \cq(t,x)} e^{\beta g_\ve(q_0) + \beta g_\ve(r_0)}\ee(Y_q(t,x) Y_r(t,y))\notag\\
&\le   \frac{C_1 e^{C_2\ve|x|}}{(2d)^{2t}}\sum_{q, r\in \cq(t,x)} e^{C_3\ve(|q_0-x| +|r_0-x|)} \mu(\beta)^{|q\cap r|}\notag \\
&= C_1 e^{C_2\ve|x|}\ee(e^{C_3\ve(|S_t|+|S_t'|)}\mu(\beta)^{N_t}),\label{zve}
\end{align}
where $\{S_n\}_{n\ge 0}$ and $\{S_n'\}_{n\ge 0}$ are independent simple symmetric random walks started at the origin in $\zz^d$, and $N_t$ is the number of times they intersect up to time $t$. Since $\mu(\beta)<1/\rho_d$, we can choose $\gamma>1$, depending only on $\beta$, $\sigma$, and $d$, such that $\mu(\beta)^\gamma <1/\rho_d$. Let $\gamma' := 2\gamma/(\gamma-1)$, so that $2/\gamma' + 1/\gamma =1$. Applying H\"older's inequality, and the displays \eqref{expsn} and~\eqref{npmf}, we get
\begin{align*}
&\ee(e^{C_3\ve(|S_t|+|S_t'|)}\mu(\beta)^{N_t}) \\
&\le [\ee(e^{C_3\ve \gamma'|S_t|})]^{1/\gamma'} [\ee(e^{C_3\ve \gamma'|S_t'|})]^{1/\gamma'} [\ee(\mu(\beta)^{\gamma N_t})]^{1/\gamma}\\
&\le C_4 e^{C_5 \ve^2 t}. 
\end{align*}
Plugging this into \eqref{zve}, we get
\begin{align}\label{zve2}
\ee(Z_\ve(t,x)^2) \le \psi_1. 
\end{align}
On the other hand, since $\ee(Y_q(t,x)) = 1$  and $g_\ve(q_0)-g_\ve(x)\ge -C\ve|q_0-x|$ for any $q\in \cq(t,x)$, we have
\begin{align}
\ee(Z_\ve(t,x)) &= \frac{1}{(2d)^t}\sum_{q\in \cq(t,x)} e^{\beta g_\ve(q_0)} \notag\\
&\ge C_1 e^{-C_2\ve |x|} \ee(e^{-C_3\ve |S_t|}). \label{zvetx}
\end{align}
Using \eqref{expminus} in \eqref{zvetx} gives the lower bound
\begin{align}\label{zve1}
\ee(Z_\ve(t,x)) &\ge 1/\psi_2. 
\end{align}
Using \eqref{zve2} and \eqref{zve1} and the Paley--Zygmund second moment inequality, we get
\begin{align}\label{ezve1}
\pp(Z_\ve(t,x) \ge1/\psi_3)\ge 1/\psi_4. 
\end{align}
Like in the proof of Lemma \ref{techlmm1}, we get
\begin{align*}
|\nabla Z_\ve(t,x)|^2 &= \frac{\beta^2 }{(2d)^{2t}}\sum_{q,r\in \cq(t,x)} e^{\beta(g_\ve(q_0)+g_\ve(r_0))} |q\cap r|Y_q(t,x)Y_r(t,x),
\end{align*}
which, using \eqref{yqyr}, gives 
\begin{align*}
\ee|\nabla Z_\ve(t,x)|^2  &= \beta^2 \ee(e^{\beta (g_\ve(S_t+x) + g_\ve(S_t'+x))} N_t \mu(\beta)^{N_t}).
\end{align*}
Now proceeding as in the proof of \eqref{zve2} above, we get 
\begin{align}\label{ezve2}
\ee|\nabla Z_\ve(t,x)|^2\le \psi_5.
\end{align}
Therefore, as in the proof of Lemma \ref{techlmm1}, we have that for any $a\in \rr$ and $b>0$,
\begin{align*}
&\pp(F_\ve(t,x) \ge a, \, |\nabla F_\ve(t,x)|\le b) \\
&\ge \pp(Z_\ve(t,x)\ge e^a) - \pp(|\nabla Z_\ve(t,x)| > be^a)\\
&\ge \pp(Z_\ve(t,x)\ge e^a) - \frac{\psi_5}{b^2e^{2a}}.
\end{align*}
Thus, choosing $a =  -\log \psi_3$ and $b = \sqrt{2\psi_4\psi_5}\psi_3$, and using \eqref{ezve1} and \eqref{ezve2} we get
\begin{align*}
\pp(F_\ve(t,x) \ge a, \, |\nabla F_\ve(t,x)|\le b) &\ge  1/2\psi_4.
\end{align*}
Again, by computing second derivatives, it is easy to verify that $F_\ve(t,x)$ is a convex function of the noise variables. By Proposition \ref{tailprop}, this shows that for all $t>0$,
\begin{align*}
\pp(F_\ve(t,x) \le - \psi_6 - t) &\le 2e^{-t^2/\psi_7 }.
\end{align*}
As in the proof of Lemma \ref{techlmm1}, this implies that $\ee(e^{-\theta F_\ve(t,x)}) \le \psi_8$.  
This completes the proof of the lemma.
\end{proof}

\section{The main argument}
This section contains the main body of the proof of Theorem \ref{polymerthm}. We will continue using the notations introduced in the previous sections. For the reader's convenience, the section is divided into small subsections. Choose two integers $1\le s<t$. These integers will be fixed throughout this section. We will also fix $x\in \rr^d$, $\ve\in (0,1)$, and $\beta>0$ such that $\mu(\beta) < 1/\rho_d$.

In this section, $C,C_1,C_2,\ldots$ will denote positive constants whose values may depend only on $\beta$, $d$, $\sigma$, $g$, and the sum $\ve|x|+\ve(t-s)+\ve \sqrt{t}+ \ve^2 t$. Furthermore, we will require that the dependence on the last item is increasing and continuous in nature. The reason is that we will later send $\ve$ to $0$, varying  $s$, $t$ and $x$ such that this quantity remains bounded. If a constant depends on some additional parameter $\theta$, we will denote it by $C(\theta)$. 

\subsection{The renormalized partition function}
Let $\cq(s,t,x)$ be the set of all nearest-neighbor paths $q = (q_s, q_{s+1},\ldots,q_t)$ with $q_t =x$. For $q\in \cq(s,t,x)$, let
\[
Y_q(s,t,x) := \frac{ \exp\bigl(\beta \sum_{i=s+1}^t \xi_{i,q_i}\bigr)}{m(\beta)^{t-s}},
\]
and define
\[
Y(s,t,x) := \frac{1}{(2d)^{t-s}}\sum_{q\in \cq(s,t,x)} Y_q(s,t,x).
\]
Let $D(s,t,x)$ be the set of possible values of $q_s$ for $q\in \cq(s,t,x)$. For each $y\in D(s,t,x)$, let $\cq(s,t,x,y)$ for the set of $q\in \cq(s,t,x)$ with $q_s=y$, and define
\[
Y(s,t,x,y) := \frac{1}{(2d)^{t-s}}\sum_{q\in \cq(s,t,x,y)} Y_q(s,t,x).
\]
Define
\[
\zeta(s,t,x,y) := \frac{Y(s,t,x,y)}{Y(s,t,x)}. 
\]
Note that $\zeta(s,t,x,\cdot)$ is a probability mass function on $D(s,t,x)$. Next, define
\begin{align*}
W_\ve(s,t,x) &= \frac{Z_\ve(t,x)}{Y(s,t,x)}. 
\end{align*}
We will refer to $W_\ve(s,t,x)$ as the {\it renormalized partition function} (as opposed to $Z_\ve(t,x)$, which is the unrenormalized partition function). Our goal in this section is to show that this random variable is approximately equal to the deterministic quantity $G_\ve(s,x)$ with high probability. The first step is the following lemma, which shows that $W_\ve(s,t,x)$ is a weighted average of $Z_\ve(s,y)$ over $y\in D(s,t,x)$.
\begin{lmm}\label{wvelmm}
We have 
\begin{align*}
W_\ve(s,t,x) &= \sum_{y\in D(s,t,x)} \zeta(s,t,x,y) Z_\ve(s,y). 
\end{align*}
\end{lmm}
\begin{proof}
Note that any $q= (q_0,\ldots,q_t)\in \cq(t,x)$ can be broken into two parts $(q_s,\ldots,q_t)\in \cq(s,t,x)$ and $(q_0,\ldots,q_s)\in \cq(s,q_s)$. Thus,
\begin{align*}
&Z_\ve(t,x) = \frac{1}{(2d)^t}\sum_{q\in \cq(t,x)} e^{\beta g_\ve(q_0)} \frac{ \exp\bigl(\beta \sum_{i=1}^t \xi_{i,q_i}\bigr)}{m(\beta)^t}\\
&=  \frac{1}{(2d)^t}\sum_{q\in \cq(s,t,x)}  \sum_{r\in \cq(s,q_s)} e^{\beta g_\ve(r_0)} \frac{ \exp\bigl(\beta \sum_{i=s+1}^t \xi_{i,q_i}\bigr)}{m(\beta)^{t-s}}\frac{ \exp\bigl(\beta \sum_{i=1}^s \xi_{i,r_i}\bigr)}{m(\beta)^{s}}\\
&=  \frac{1}{(2d)^{t-s}}\sum_{q\in \cq(s,t,x)}  \frac{ \exp\bigl(\beta \sum_{i=s+1}^t \xi_{i,q_i}\bigr)}{m(\beta)^{t-s}}Z_\ve(s, q_s)\\
&= \sum_{y\in D(s,t,x)} Y(s,t,x,y) Z_\ve(s, y). 
\end{align*}
Dividing both sides by $Y(s,t,x)$ gives the desired identity.
\end{proof}
\subsection{Expectation of the renormalized partition function}
Let $\mf$ be the $\sigma$-algebra generated by the random variables $\{\xi_{u,y}: s+1\le u\le t, \ y\in \zz^d\}$. 
Let $\ee'$, $\pp'$, $\var'$ and $\cov'$  denote conditional expectation, conditional probability, conditional variance and conditional covariance given $\mf$. The following lemma shows that $\ee'(W_\ve(s,t,x))$ is a weighted average of $G_\ve(s,y)$ over $y\in D(s,t,x)$.
\begin{lmm}\label{epwve}
We have
\[
\ee'(W_\ve(s,t,x)) =  \sum_{y\in D(s,t,x)} \zeta(s,t,x,y)G_\ve(s,y).
\]
\end{lmm}
\begin{proof}
Since $\zeta(s,t,x,y)$ is $\mf$-measurable and $Z_\ve(s,y)$ is independent of $\mf$ for any $y$, we have
\begin{align*}
\ee'(W_\ve(s,t,x)) &=  \sum_{y\in D(s,t,x)} \zeta(s,t,x,y) \ee(Z_\ve(s,y))\\
&=  \sum_{y\in D(s,t,x)} \zeta(s,t,x,y)\biggl(\frac{1}{(2d)^s}\sum_{q\in \cq(s,y)} e^{\beta g_\ve(q_0)}\biggr).
\end{align*}
To complete the proof, just note that the term inside the bracket on the right side is nothing but $G_\ve(s,y)$.
\end{proof}
As a corollary of the above lemma and Lemma \ref{glip}, we obtain the following result, which shows that $\ee'(W_\ve(s,t,x))$ is close to $G_\ve(s,x)$ if $t-s = o(1/\ve)$.
\begin{cor}\label{epwcor}
We have
\[
|\ee'(W_\ve(s,t,x)) - G_\ve(s,x)| \le C\ve (t-s) 
\]
\end{cor}
\begin{proof}
Since $\zeta(s,t,x,\cdot)$ is a probability mass function on $D(s,t,x)$, we have 
\begin{align*}
|\ee'(W_\ve(s,t,x)) - G_\ve(s,x)| &= \biggl|\sum_{y\in D(s,t,x)} \zeta(s,t,x,y) (G_\ve(s,y)-G_\ve(s,x))\biggr|\\
&\le \sum_{y\in D(s,t,x)} \zeta(s,t,x,y) |G_\ve(s,y)-G_\ve(s,x)|.
\end{align*}
To complete the proof, apply Lemma \ref{glip} to bound $|G_\ve(s,y)-G_\ve(s,x)|$ and use the fact that $|x-y|\le t-s$ for $y\in D(s,t,x)$.
\end{proof}

\subsection{Variance of the renormalized partition function}\label{varsection}
We will now show that the renormalized partition function has small conditional variance if $\ve$ is small. The first step is to get a bound for covariances. 
\begin{lmm}\label{covlmm}
For any $y,y'\in D(s,t,x)$,
\[
\cov(Z_\ve(s,y), Z_\ve(s,y')) \le C_1 \kappa(y,y')^{C_2}.
\]
\end{lmm}
\begin{proof}
By the Lipschitz property of $g$ and the fact that $|y-x|\le t-s$ for any $y\in D(s,t,x)$, we have that for any $y,y'\in D(s,t,x)$, 
\begin{align*}
&\cov(Z_\ve(s,y), Z_\ve(s,y')) \\
&= \frac{1}{(2d)^{2s}}\sum_{q\in \cq(s,y),\, r\in \cq(s,y')} e^{\beta (g_\ve(q_0)+g_\ve(r_0))} \cov(Y_q(s,y), Y_r(s,y'))\\
&\le  \frac{C_1}{(2d)^{2s}}\sum_{q\in \cq(t,y),\, r\in \cq(t,y')} e^{C_2\ve (|q_0-y|+|r_0-y'|)}(\mu(\beta)^{|q\cap r|} - 1)\\
&= C_1 \ee(e^{C_2\ve (|S_t -y|+|S_t'-y'|)}(\mu(\beta)^{N_t^{y,y'}} - 1)),
\end{align*}
where $\{S_n\}_{n\ge 0}$ and $\{S_n'\}_{n\ge 0}$ denote independent simple symmetric random walks started from $y$ and $y'$, and $N_t^{y,y'}$ is the number of times they intersect up to time $t$. Since $\mu(\beta)<1/\rho_d$, we can choose $\gamma>1$, depending only on $\beta$, $\sigma$, and $d$, such that $\mu(\beta)^\gamma <1/\rho_d$. Let $\gamma' := 2\gamma/(\gamma-1)$, so that $2/\gamma' + 1/\gamma =1$. Applying H\"older's inequality, \eqref{expsn}, and \eqref{covprep}, we get
\begin{align*}
&\ee(e^{C_2\ve (|S_t -y|+|S_t'-y'|)}(\mu(\beta)^{N_t^{y,y'}} - 1))\\
&\le [\ee(e^{C_2\gamma'\ve |S_t -y|})]^{1/\gamma'}[\ee(e^{C_2\gamma'\ve |S_t' -y'|})]^{1/\gamma'}[\ee|\mu(\beta)^{N_t^{y,y'}} - 1|^\gamma]^{1/\gamma}\\
&\le C_3 [\ee(\mu(\beta)^{\gamma N_\infty})]^{1/\gamma} \kappa(y,y')^{1/\gamma},
\end{align*}
where $N_\infty$ is the number of intersections of two independent simple symmetric random walks in $\zz^d$, started at the origin. Applying \eqref{npmf} completes the proof.
\end{proof}
The next lemma gives an upper bound for the $\var'(W_\ve(s,t,x))$, obtained using the bound on covariances from Lemma \ref{covlmm}.
\begin{lmm}\label{vpwlmm}
We have
\begin{align*}
\var'(W_\ve(s,t,x)) &\le C_1\sum_{y,y'\in D(s,t,x)} \zeta(s,t,x,y)\zeta(s,t,x,y')\kappa(y,y')^{C_2}.
\end{align*}
\end{lmm} 
\begin{proof}
Since $\zeta(s,t,x,y)$ is $\mf$-measurable for each $y$, and $\{Z_\ve(s,y)\}_{y\in \zz^d}$ is independent of $\mf$, we have
\begin{align*}
&\var'(W_\ve(s,t,x)) \\
&= \sum_{y,y'\in D(s,t,x)} \zeta(s,t,x,y)\zeta(s,t,x,y') \cov(Z_\ve(s,y), Z_\ve(s,y')).
\end{align*}
The proof is now completed by invoking Lemma \ref{covlmm}.
\end{proof}
Next, we use the above lemma, and Lemma \ref{kappalmm}, to obtain a bound on the expected value of $\var'(W_\ve(s,t,x))$. 
\begin{lmm}\label{evpwlmm}
We have
\[
\ee[\var'(W_\ve(s,t,x))] \le C_1 (t-s)^{-C_2}. 
\]
\end{lmm}
\begin{proof}
Let $C_2$ be as in the statement of Lemma \ref{vpwlmm}. First, suppose that in a particular realization of the noise field, we have $Y(s,t,x)\ge (t-s)^{-\alpha}$, where $\alpha>0$ is a constant that we will choose later. Then we have
\begin{align*}
&\sum_{y,y'\in D(s,t,x)} \zeta(s,t,x,y)\zeta(s,t,x,y')\kappa(y,y')^{C_2} \\
&\le (t-s)^{2\alpha} \sum_{y,y'\in D(s,t,x)} Y(s,t,x,y)Y(s,t,x,y')\kappa(y,y')^{C_2}.
\end{align*}
Now observe that
\begin{align*}
&\ee\biggl(\sum_{y,y'\in D(s,t,x)} Y(s,t,x,y)Y(s,t,x,y')\kappa(y,y')^{C_2}\biggr) \\
&= \frac{1}{(2d)^{2(t-s)}}\sum_{y,y'\in D(s,t,x)} \sum_{\substack{q\in \cq(s,t,x,y),\\ q'\in \cq(s,t,x,y')}} \ee(Y_q(s,t,x,y)Y_{q'}(s,t,x,y'))\kappa(y,y')^{C_2}\\
&=  \frac{1}{(2d)^{2(t-s)}}\sum_{y,y'\in D(s,t,x)} \sum_{\substack{q\in \cq(s,t,x,y),\\ q'\in \cq(s,t,x,y')}} \mu(\beta)^{|q\cap q'|}\kappa(y,y')^{C_2}\\
&= \frac{1}{(2d)^{2(t-s)}}\sum_{q,q'\in \cq(s,t,x)} \mu(\beta)^{|q\cap q'|} \kappa(q_s,q_s')^{C_2}. 
\end{align*}
The expression in the last line equals
\[
\ee(\mu(\beta)^{N_{t-s}} \kappa(S_{t-s}, S_{t-s}')^{C_2}),
\]
where $\{S_n\}_{n\ge 0}$ and $\{S_n'\}_{n\ge 0}$ are independent simple symmetric random walks started at the origin in $\zz^d$, and $N_{t-s}$ is the number of times they intersect up to time $t$. 
Since $\mu(\beta)<1/\rho_d$, we can choose $\delta >1$, depending only on $\beta$, $\sigma$, and $d$, such that $\mu(\beta)^\delta <1/\rho_d$. Let $\delta' := \delta/(\delta-1)$, so that $1/\delta' + 1/\delta =1$. Applying H\"older's inequality, we get
\begin{align*}
\ee(\mu(\beta)^{N_{t-s}} \kappa(S_{t-s}, S_{t-s}')^{C_2}) &\le [\ee(\mu(\beta)^{\delta N_{t-s}})]^{1/\delta} [\ee(\kappa(S_{t-s}, S_{t-s}')^{C_2\delta'})]^{1/\delta'}\\
&\le [\ee(\mu(\beta)^{\delta N_\infty})]^{1/\delta} [\ee(\kappa(S_{t-s}, S_{t-s}')^{C_2\delta'})]^{1/\delta'}
\end{align*}
Applying \eqref{npmf} and Lemma~\ref{kappalmm} to the right side gives
\begin{align*}
\ee(\mu(\beta)^{N_{t-s}} \kappa(S_{t-s}, S_{t-s}')^{C_2}) &\le C_3 (t-s)^{-C_4}. 
\end{align*}
Combining all of the above observations, and using Lemma \ref{vpwlmm}, we get that 
\begin{align*}
\ee[\var'(W_\ve(s,t,x)); Y(s,t,x)\ge (t-s)^{-\alpha}] &\le C_5 (t-s)^{-C_6+2\alpha}. 
\end{align*}
On the other hand, since $\kappa$ is uniformly bounded by $1$, and $\zeta(s,t,x,\cdot)$ is a probability mass function, Lemma \ref{vpwlmm} implies that $\var'(W_\ve(s,t,x))\le C_7$. Combining this with the above bound, we get
\begin{align*}
\ee[\var'(W_\ve(s,t,x))] &\le C_5 (t-s)^{-C_6+2\alpha} + C_7\pp(Y(s,t,x)< (t-s)^{-\alpha}). 
\end{align*}
Now note that $Y(s,t,x)$ has the same distribution as $Y(t-s,x)$. Therefore by Lemma \ref{techlmm1}, 
\begin{align*}
\pp(Y(s,t,x)< (t-s)^{-\alpha}) &= \pp(e^{-F(t-s,x)} > (t-s)^\alpha)\\
&\le (t-s)^{-\alpha} \ee(e^{-F(t-s,x)})\\
&\le C_8(t-s)^{-\alpha}.
\end{align*}
Plugging this bound into the preceding display, and choosing $\alpha = C_6/4$, we get the desired bound. 
\end{proof}
\subsection{Concentration of the renormalized height function}
The goal of this subsection is to show that the random variable $\log W_\ve(s,t,x)$, which we call the {\it renormalized height} of our random surface at $(t,x)$, is close to the deterministic quantity $\log G_\ve(s,x)$ in $L^2$ distance, if $1\ll t-s \ll \ve^{-1}$. The proof has several steps. The first step is the following, which shows that $W_\ve(s,t,x)$ is close to $G_\ve(s,x)$ in $L^2$ distance.
\begin{lmm}\label{wve2lmm}
We have 
\begin{align*}
\ee[(W_\ve(s,t,x)-G_\ve(s,x))^2]&\le  C_1 (\ve^2 (t-s)^2 + (t-s)^{-C_2}). 
\end{align*}
\end{lmm}
\begin{proof}
Simply observe that
\begin{align*}
&\ee[(W_\ve(s,t,x)-G_\ve(s,x))^2] \\
&= \ee[\var'(W_\ve(s,t,x)] + \ee[(\ee'(W_\ve(s,t,x)) - G_\ve(s,x))^2],
\end{align*}
and apply Lemma \ref{evpwlmm} and Corollary \ref{epwcor} to bound the two terms on the right.
\end{proof}
Next, we show that $\log W_\ve(s,t,x)$ is close to $\log G_\ve(s,x)$ in $L^{1/2}$ distance. 
\begin{lmm}\label{loglmm}
We have
\begin{align*}
\ee|\log W_\ve(s,t,x)-\log G_\ve(s,x)|^{1/2} &\le C_1(\sqrt{\ve(t-s)} + (t-s)^{-C_2}). 
\end{align*}
\end{lmm}
\begin{proof}
First, note that 
\begin{align*}
&|\log W_\ve(s,t,x)-\log G_\ve(s,x)|^{1/2}  \\
&\le |W_\ve(s,t,x)- G_\ve(s,x)|^{1/2}(W_\ve(s,t,x)^{-1/2} + G_\ve(s,x)^{-1/2}).
\end{align*}
By Lemma \ref{techlmm2}, the observation that $Y(s,t,x)$ has the same law as $Y(t-s,x)$, and the bound \eqref{y2bound}, we get
\begin{align*}
\ee(W_\ve(s,t,x)^{-1}) &= \ee\biggl(\frac{Y(s,t,x)}{Z_\ve(t,x)}\biggr) \\
&\le [\ee(Y(t-s,x)^2)]^{1/2} [\ee(Z_\ve(t,x)^{-2})]^{1/2}\le C.
\end{align*}
Also, by Lemma \ref{gbounds}, $G_\ve(s,x)^{-1/2}\le C$, and by Lemma \ref{wve2lmm}, 
\[
\ee|W_\ve(s,t,x)-G_\ve(s,x))| \le C_1(\ve (t-s) + (t-s)^{-C_2}).
\]
It is now easy to complete the proof by combining all of the above information and applying the Cauchy--Schwarz inequality.
\end{proof}
We want to improve Lemma \ref{loglmm} to an $L^2$ bound. For that, we need the following moment bound.
\begin{lmm}\label{logbound}
For any $\alpha \ge 1$, 
\[
\ee|\log W_\ve(s,t,x)|^\alpha\le C(\alpha).
\]
\end{lmm}
\begin{proof}
First, note that since $Y(s,t,x)$ has the same distribution as $Y(t-s,x)$,
\begin{align*}
\ee|\log W_\ve(s,t,x)|^\alpha &\le C(\alpha) (\ee|\log Z_\ve(t,x)|^\alpha + \ee|\log Y(s,t,x)|^\alpha)\\
&= C(\alpha) (\ee|F_\ve(t,x)|^\alpha + \ee|F(t-s,x)|^\alpha). 
\end{align*}
Now, by Lemma \ref{techlmm1},
\begin{align*}
 \ee|F(t-s,x)|^\alpha &\le C(\alpha)\ee(e^{|F(t-s,x)|})\\
 &\le C(\alpha)(\ee(e^{F(t-s,x)}) + \ee(e^{-F(t-s,x)}))\\
 &\le C(\alpha)(\ee(Y(t-s,x)) + C).
\end{align*}
But, as already observed in the proof of Lemma \ref{techlmm1}, $\ee(Y(t-s,x))=1$. Thus, 
\[
\ee|F(t-s,x)|^\alpha \le C(\alpha).
\]
Similarly, by Lemma \ref{techlmm2}, 
\begin{align*}
\ee|F_\ve(t,x)|^\alpha &\le C(\alpha) \ee(e^{|F_\ve(t,x)|})\\
&\le C(\alpha) (\ee(e^{F_\ve(t,x)}) + \ee(e^{-F_\ve(t,x)}))\\
&\le C(\alpha) (\ee(Z_\ve(t,x)) + C). 
\end{align*}
By \eqref{expsn} and the Lipschitz property of $g$, 
\begin{align*}
\ee(Z_\ve(t,x)) &= \frac{1}{(2d)^t}\sum_{q\in \cq(t,x)} e^{\beta g_\ve(q_0)} \\
&\le C_1 e^{C_2\ve |x|} \ee(e^{C_3\ve |S_t|})\le C_4. 
\end{align*}
Combining the above observations, the proof is complete. 
\end{proof}
We are now ready to prove the main result of this subsection, which is the following.
\begin{lmm}\label{loglmm2}
We have
\begin{align*}
\ee[(\log W_\ve(s,t,x)-\log G_\ve(s,x))^2] &\le C_1((\ve(t-s))^{1/4} + (t-s)^{-C_2}). 
\end{align*}
\end{lmm}
\begin{proof}
For a nonnegative random variable $X$, note that by the Cauchy--Schwarz inequality,
\[
\ee(X^2) = \ee(X^{1/4} X^{7/4}) \le \sqrt{\ee(X^{1/2}) \ee(X^{7/2})}. 
\]
Applying this to the random variable $|\log W_\ve(s,t,x)-\log G_\ve(s,x)|$, and using Lemma \ref{loglmm}, Lemma \ref{logbound}, and Lemma \ref{gbounds}, we get the desired result.
\end{proof}

\subsection{Concentration of local averages}
Let $1\le r\le \ve^{-1}$ be a real number. Let $B(x,r)$ be the set of all points in $\zz^d$ that are within Euclidean distance $r$ from $x$. Define 
\begin{align}\label{xdef}
X := \frac{1}{|B(x,r)|} \sum_{y\in B(x,r)} F_\ve(t,y).
\end{align}
Let $\alpha := \log G_\ve(t,x) + \ee(\log Y(t-s,0))$. The following result shows that $X$ is close to $\alpha$ in $L^2$ distance if $1\ll t-s\ll r\ll \ve^{-1}$. 
\begin{lmm}\label{xlmm}
Let $r$, $X$ and $\alpha$ be as above. Then
\begin{align*}
\ee[(X-\alpha)^2] &\le C_1 r^{-d}(t-s)^d+C_1(\ve(t-s))^{1/4}\\
&\qquad +C_1(t-s)^{-C_2} +C_1\ve^2(t-s) + C_1\ve^2r^2.
\end{align*}
\end{lmm}
\begin{proof}
Note that paths in $\cq(s,t,x)$ and $\cq(s,t,y)$ do not intersect if $|x-y|>t-s$. Thus, $Y(s,t,x)$ and $Y(s,t,y)$ are independent random  variables if $|x-y|>t-s$. On the other hand, from the proof of Lemma \ref{logbound}, we know that $\ee|\log Y(s,t,x)|^2\le C$. Combining these two observations, we get that 
\begin{align*}
&\var\biggl(\frac{1}{|B(x,r)|}\sum_{y\in B(x,r)}\log Y(s,t,y)\biggr)\\
&=  \frac{1}{|B(x,r)|^2}\sum_{y,y'\in B(x,r)} \cov(\log Y(s,t,y), \log Y(s,t,y')) \\
&=  \frac{1}{|B(x,r)|^2} \sum_{y,y'\in B(x,r), \, |y-y'|\le t-s} \cov(\log Y(s,t,y), \log Y(s,t,y')) \\
&\le Cr^{-d}(t-s)^d.
\end{align*}
Define 
\[
\gamma := \ee(\log Y(t-s,0)) + \frac{1}{|B(x,r)|} \sum_{y\in B(x,r)} \log G_\ve(s,y).
\]
Recall that $F_\ve(t,y)  = \log Y(s,t,y)+ \log W_\ve(s,t,y)$. Since $Y(s,t,x)$ has the same distribution as $Y(t-s,x)$, we have $\ee(\log Y(s,t,y)) =  \ee(\log Y(t-s,0))$ for all $y$. Thus, 
\begin{align*}
\ee[(X-\gamma)^2] &= \ee\biggl[ \biggl(\frac{1}{|B(x,r)|} \sum_{y\in B(x,r)} (\log W_\ve(s,t,y) -\log G_\ve(s,y)) \\
&\qquad + \frac{1}{|B(x,r)|} \sum_{y\in B(x,r)} (\log Y(s,t,y) -  \ee(\log Y(s,t,y)))\biggr)^2\biggr]\\
&\le 2\ee\biggl[\biggl(\frac{1}{|B(x,r)|} \sum_{y\in B(x,r)} (\log W_\ve(s,t,y)-\log G_\ve(s,y))\biggr)^2\biggr]\\
&\qquad + 2\var\biggl(\frac{1}{|B(x,r)|} \sum_{y\in B(x,r)} \log Y(s,t,y)\biggr).
\end{align*}
By Lemma \ref{loglmm2},
\begin{align*}
&\ee\biggl[\biggl(\frac{1}{|B(x,r)|} \sum_{y\in B(x,r)} (\log W_\ve(s,t,y)-\log G_\ve(s,y))\biggr)^2\biggr]\\
&\le\frac{1}{|B(x,r)|} \sum_{y\in B(x,r)}  \ee[(\log W_\ve(s,t,y)-\log G_\ve(s,y))^2]\\
&\le C_1((\ve(t-s))^{1/4} + (t-s)^{-C_2}).
\end{align*}
Next, note that
\begin{align*}
|\alpha - \gamma| &\le  |\log G_\ve(s,x)-\log G_\ve(t,x)| \\
&\qquad + \frac{1}{|B(x,r)|} \sum_{y\in B(x,r)} |\log G_\ve(s,y)-\log G_\ve(s,x)|. 
\end{align*}
By Lemma \ref{gtbound} and Lemma \ref{gbounds}, we have
\begin{align*}
&|\log G_\ve(s,x)-\log G_\ve(t,x)|  \\
&\le |G_\ve(s,x)- G_\ve(t,x)| (G_\ve(s,x)^{-1}+G_\ve(t,x)^{-1})\\
&\le C\ve \sqrt{t-s}. 
\end{align*}
Similarly, by Lemma \ref{glip} and Lemma \ref{gbounds}, we have that for any $y\in B(x,r)$,
\begin{align*}
&|\log G_\ve(s,y)-\log G_\ve(s,x)|  \\
&\le |G_\ve(s,y)- G_\ve(s,x)| (G_\ve(s,y)^{-1}+G_\ve(s,x)^{-1})\\
&\le C\ve |x-y|\le C\ve r. 
\end{align*}
Combining all of the above observations, we get the desired bound. 
\end{proof}

\section{Proof of Proposition \ref{etaprop}}
Proposition \ref{etaprop} can be proved by a martingale argument, but we already have all the ingredients to give a proof that yields a rate of convergence. The following proposition is a better version of Proposition \ref{etaprop}. 
\begin{prop}\label{etaprop2}
Take any $\beta>0$ such that $\mu(\beta) < 1/\rho_d$. Let $Y(t,x)$ be defined as in equation \eqref{ydef}. For each positive integer $t$, let $\eta(\beta, t) := \ee(\log Y(t,0))$. Then $\eta(\beta) := \lim_{t\to\infty} \eta(\beta, t)$  exists and is finite. Moreover, for any $t$,
\[
|\eta(\beta,t) - \eta(\beta)| \le C_1 t^{-C_2},
\]
where $C_1$ and $C_2$ are positive constants that depend only on $\beta$, $d$ and $\sigma$.
\end{prop}
\begin{proof}
Take any $1\le u<t$, and let $s:= t-u$. Take any $\ve \in (0,1)$. We will use Lemma \ref{loglmm2} with $g\equiv 0$. When $g\equiv 0$, we have $G_\ve(t,x)=1$ for any $\ve$, $t$ and $x$. Also, we have $Z_\ve(t,x)=Y(t,x)$. Thus, Lemma \ref{loglmm2} gives 
\begin{align*}
|\eta(\beta, t) - \eta(\beta, u)| &= |\ee(\log Y(t,0)) - \ee(\log Y(t-s,0))|\\
&= |\ee(\log Y(t,0)) - \ee(\log Y(s,t,0))|\\ 
&\le \ee|\log Y(t,0)-\log Y(s,t,0)|\\
&\le \sqrt{\ee[(\log Y(t,0)-\log Y(s,t,0))^2]} \\
&\le C_1((\ve(t-s))^{1/8} + (t-s)^{-C_2}), 
\end{align*}
where $C_1$ and $C_2$ are positive constants that depend only on $\beta$, $d$, $\sigma$, and the sum $\ve|x|+\ve(t-s)+\ve \sqrt{t}+ \ve^2 t$. Moreover the dependence on the sum is increasing and continuous. But the left side does not depend on $\ve$. So we can take $\ve \to 0$ and get the bound $|\eta(\beta,t)-\eta(\beta, u)|\le C_3 u^{-C_4}$, where $C_3$ and $C_4$ depend only on $\beta$, $d$, and $\sigma$. This suffices to prove both claims of the lemma.
\end{proof}

\section{Proof of Theorem \ref{polymerthm}}
Since the logarithm of any moment generating function is convex, 
\[
\frac{m'(\beta)}{m(\beta)} = \frac{d}{d\beta} \log m(\beta)
\]
is an increasing function of $\beta$. Thus,
\[
\frac{d}{d\beta} \log \mu(\beta) = 2\biggl(\frac{m'(2\beta)}{m(2\beta)} - \frac{m'(\beta)}{m(\beta)}\biggr)\ge 0.
\]
This shows that $\mu$ is a non-decreasing function, which implies that $\mu(\beta)<1/\rho_d$ for all $\beta\in (0,\beta_0)$. Thus, all the lemmas proved in the preceding sections are applicable when $\beta \in (0,\beta_0)$.

Take any $(t,x)\in \rr_{>0}\times \rr^d$. Let $t_\ve$ and $x_\ve$ be defined as in equation \eqref{tvedef} and let $\tf^{(\ve)}$ be defined as in equation \eqref{tfvedef}. Let $\{r_\ve\}_{\ve >0}$ be any collection of positive real numbers such that $1\ll r_\ve \ll \ve^{-1}$ as $\ve \to 0$. Define $s_\ve := t_\ve - [\sqrt{r_\ve}]$. Define $F_\ve$ as in Section \ref{concheightsec}, so that for any $y\in \zz^d$,
\begin{align}\label{fFeq}
F_\ve(t_\ve,y) = \beta f_\ve(t_\ve,y) - t_\ve \log(2dm(\beta)). 
\end{align}
In analogy with \eqref{xdef}, let
\begin{align*}
X_\ve := \frac{1}{|B(x_\ve,r_\ve)|} \sum_{y\in B(x_\ve,r_\ve)} F_\ve(t_\ve,y),
\end{align*}
so that, by \eqref{fFeq},
\begin{align}\label{fXeq}
\tf^{(\ve)}(t,x) &= \beta^{-1}X_\ve - \beta^{-1}\eta(\beta). 
\end{align}
Similarly, let $\alpha_\ve := \log G_\ve(t_\ve,x_\ve) + \ee(\log Y(t_\ve-s_\ve,0))$. Now observe that, as $\ve \to 0$,
\begin{itemize}
\item $\ve|x_\ve|+\ve(t_\ve-s_\ve)+\ve \sqrt{t_\ve}+ \ve^2 t_\ve$ remains uniformly bounded above,
\item $r_\ve^{-d}(t_\ve-s_\ve)^d \to 0$,
\item $\ve(t_\ve -s_\ve)\to 0$,
\item $t_\ve -s_\ve \to \infty$, and 
\item $\ve r_\ve \to 0$. 
\end{itemize}
Using all of the above observations, and the bound from Lemma \ref{xlmm}, we get that $X_\ve -\alpha_\ve\to 0$ in $L^2$ as $\ve \to 0$. But by Lemma \ref{glim} and Proposition \ref{etaprop2}, 
\begin{align*}
\lim_{\ve \to 0} \alpha_\ve &= \beta h(t,x) + \eta(\beta). 
\end{align*}
Therefore, by \eqref{fXeq}, $\tf^{(\ve)}(t,x)\to h(t,x)$ in $L^2$ as $\ve \to 0$. 

\section{Proof of Corollary \ref{polymercor}}
In this proof, $C, C_1,C_2,\ldots$ will denote positive constants that may depend only on $\beta$, $d$, $\sigma$, and $g$, whose values may change from line to line. As in the proof of Theorem \ref{polymerthm}, choose $\{r_\ve\}_{\ve >0}$ such that, as $\ve \to 0$, $1\ll r_\ve \ll \ve^{-1}$. Define $\tf^{(\ve)}$ as in \eqref{tfvedef}. By the Cauchy--Schwarz inequality, for any $t>0$,
\begin{align*}
&\ee\biggl[\biggl(\int \tf^{(\ve)}(t,x) \phi(x)dx - \int h(t,x)\phi(x) dx\biggr)^2\biggr] \\
&\le \int \ee[(\tf^{(\ve)}(t,x)-h(t,x))^2] |\phi(x)|dx \int |\phi(x)|dx. 
\end{align*}
By Theorem \ref{polymerthm}, $\ee[(\tf^{(\ve)}(t,x)-h(t,x))^2] \to 0$ for any $x$ as $\ve \to 0$. Moreover, by the bound from Lemma \ref{xlmm}, we see that this convergence is uniform over any compact set. Since $\phi$ has compact support and is bounded, this proves that as $\ve \to 0$, 
\[
\int \tf^{(\ve)}(t,x)\phi(x)dx \stackrel{L^2}{\to} \int h(t,x)\phi(x)dx.
\]
Thus, it suffices to show that
\begin{align}\label{ftf}
\int (f^{(\ve)}(t,x)-\tf^{(\ve)}(t,x))\phi(x)dx \stackrel{L^2}{\to} 0.
\end{align}
To prove this, first note that 
\begin{align*}
\tf^{(\ve)}(t,x) - f^{(\ve)}(t,x) &= \frac{1}{|B(x_\ve, r_\ve)|} \sum_{y\in B(x_\ve, r_\ve)} (f_\ve(t_\ve,y)-f_\ve(t_\ve,x_\ve))\\
&= \frac{1}{|B(0, r_\ve)|} \sum_{y\in B(0, r_\ve)} (f_\ve(t_\ve,x_\ve +y)-f_\ve(t_\ve,x_\ve)).
\end{align*}
Recall that $x_\ve = [\ve^{-1}x]$. Thus, by the change of variable $z = x +\ve y$,
\begin{align*}
\int f_\ve(t_\ve,x_\ve +y) \phi(x) dx &= \int f_\ve(t_\ve, [\ve^{-1}x] +y) \phi(x) dx \\
&= \int f_\ve(t_\ve,[\ve^{-1}z - y] +y) \phi(z - \ve y) dz.
\end{align*}
Since $y\in \zz^d$, $[\ve^{-1}z - y] = [\ve^{-1}z] - y$. Thus, 
\begin{align*}
\int f_\ve(t_\ve,x_\ve +y) \phi(x) dx &= \int f_\ve(t_\ve, x_\ve) \phi(x-\ve y) dx.
\end{align*}
Recall the function $F_\ve$ defined in Section \ref{concheightsec}. The above calculations show that
\begin{align*}
&\int (f^{(\ve)}(t,x)-\tf^{(\ve)}(t,x))\phi(x)dx \\
&= \frac{1}{|B(0,r_\ve)|} \sum_{y\in B(0,r_\ve)} \int f_\ve(t_\ve, x_\ve)(\phi(x)-\phi(x-\ve y)) dx\\
&= \frac{1}{|B(0,r_\ve)|} \sum_{y\in B(0,r_\ve)} \int \beta^{-1}F_\ve(t_\ve, x_\ve)(\phi(x)-\phi(x-\ve y)) dx, 
\end{align*}
where the last identity holds because $f_\ve(t_\ve, x_\ve)$ and $\beta^{-1}F_\ve(t_\ve, x_\ve)$ differ by a constant. Thus,
\begin{align*}
&\ee\biggl[\biggl(\int (f^{(\ve)}(t,x)-\tf^{(\ve)}(t,x))\phi(x)dx\biggr)^2\biggr]\\
&\le \frac{1}{|B(0,r_\ve)|} \sum_{y\in B(0,r_\ve)} \ee\biggl[\biggl(\int \beta^{-1}F_\ve(t_\ve, x_\ve)(\phi(x)-\phi(x-\ve y)) dx\biggr)^2\biggr].
\end{align*}
Let $R$ be so large that $\phi(x) = 0$ when $|x|\ge R-1$. For each $\delta>0$, let
\[
\omega(\delta):= \sup\{|\phi(y)-\phi(z)|: |y-z|\le \delta\}. 
\]
If $\ve$ is so small that $\ve r_\ve \le 1$, then for any $y\in B(0,r_\ve)$, $|x-\ve y|\ge |x|-1$. Therefore, by the Cauchy--Schwarz inequality,
\begin{align*}
&\ee\biggl[\biggl(\int \beta^{-1}F_\ve(t_\ve, x_\ve)(\phi(x)-\phi(x-\ve y)) dx\biggr)^2\biggr]\\
&= \ee\biggl[\biggl(\int_{|x|\le R} \beta^{-1}F_\ve(t_\ve, x_\ve)(\phi(x)-\phi(x-\ve y)) dx\biggr)^2\biggr]\\
&\le C R^d\omega(\ve r_\ve)\int_{|x|\le R} \ee[F_\ve(t_\ve, x_\ve)^2] dx.
\end{align*}
From the proof of Lemma \ref{logbound}, it is easy to see that $\ee[F_\ve(t_\ve, x_\ve)^2] \le C$ uniformly over $|x|\le R$. Plugging this into the above, we get
\begin{align*}
\ee\biggl[\biggl(\int (f^{(\ve)}(t,x)-\tf^{(\ve)}(t,x))\phi(x)dx\biggr)^2\biggr] &\le CR^{2d} \omega(\ve r_\ve). 
\end{align*}
By the uniform continuity of $\phi$, the right side tends to zero as $\ve \to 0$. This proves equation \eqref{ftf}, and hence completes the proof of the corollary.

\section*{Acknowledgments}
I thank Chiranjib Mukherjee, Nikos Zygouras, and the anonymous referees for a number of helpful comments and references. 

\bibliographystyle{plainnat}
\bibliography{myrefs}

\end{document}